\documentclass[12pt]{amsart}
\usepackage[utf8]{inputenc}
\usepackage[T1]{fontenc}

\usepackage{amsmath}
\usepackage{amsthm}
\usepackage{amssymb}
\usepackage{amscd}
\usepackage{amsfonts}
\usepackage{amsbsy}
\usepackage{graphicx}
  \usepackage{mathtools,amsthm,amssymb,epsfig,graphicx,amsmath}
  \usepackage[all]{xy}

\def\R{\mathbb{R}}
\def\C{\mathbb C}
\def\Z{\mathbb Z}
\def\N{\mathbb N}
\def\P{\mathbb P}

\usepackage{mathrsfs}

\renewcommand{\epsilon}{\varepsilon}

\newtheorem{definition}{Definition}

\newtheorem{example}{Example}
\newtheorem{theorem}{Theorem} 
\newtheorem{lemma}{Lemma}
\newtheorem{proposition}{Proposition}
\newtheorem{corollary}{Corollary}
\vspace{5mm}
\begin{document}

\title[Two lectures on the center-focus problem]{Two lectures on the center-focus problem for the equation 
$\frac{dy}{dx} + \sum_{i=1}^n a_i(x) y^i = 0, 0\leq x \leq 1$\\
where $a_i$ are polynomials
}
 \author[L. Gavrilov]{Lubomir Gavrilov}
\address{
Institut de Math\'ematiques de Toulouse; UMR5219 \\
Universit\'e de Toulouse; CNRS \newline
UPS IMT, F-31062 Toulouse Cedex 9, France}
\email{lubomir.gavrilov@math.univ-toulouse.fr}
\thanks{Research partially supported by the Grant No DN 02-5 of the Bulgarian Fund “Scientific Research"}



\begin{abstract}
This is an extended version of two lectures given during the  Zagreb Dynamical Systems Workshop, October 22-26, 2018. 
\end{abstract}

\maketitle
\tableofcontents



\section{The  center-focus problem}
The plane differential system
\begin{align}
\label{system}
\dot{x} &= P(x,y), \;
\dot{y}= Q(x,y)
\end{align}
is said to have a \emph{center} at the singular point $(0,0)$, if in a sufficiently small neighbourhood of this point all orbits are closed. 
Consider the   scalar  differential equation
\begin{align}
\label{ode}
\frac{dy}{dx} + f(x,y)   = 0, x\in [0,1]
\end{align}
in which $f(x,0) =0,  \forall x\in [0,1]$. The equation (\ref{ode}) is said to have a \emph{center} at $y=0$, if all solutions $y(x)$ starting near the origin, satisfy $y(0)=y(1)$ (the interval $[0,1]$ can be replaced by any closed interval). 

\emph{Note on the terminology.} We do not specify here the category to which belong $P, Q, f$. They will be either analytic or polynomial, depending on the context. The base field will be either $\R$ or $\C$ depending on the context too. Most results will be valid for both. Thus, the definition of a center for (\ref{ode}) is the same in the real  and in the complex case. In the case of an analytic  complex plane vector field (\ref{system}) the "complex" definition of a center is less straightforward. 
We say that the origin is  a non-degenerate center, if the vector field has an analytic first integral with a Morse critical point at the origin. If this is the case, we shall also say that (\ref{system}) 
has  a Morse singular point, e.g. \cite{cerlin96}.
We recall therefore
\begin{definition}
The analytic complex  vector field (\ref{system}) is said to have a Morse singular point, if it allows an analytic first integral in a neighbourhood of this point, which has a Morse type of singularity.
\end{definition}
If (\ref{system}) has a Morse singular point, then  the linear part of (\ref{system}) is diagonalisable with non-zero eigenvalues, that is to say the singular point of the vector field is non-degenerate. 

An example is the saddle $x'=x, y'=-y$ which has an analytic first integral $xy$ of Morse type, and hence a Morse critical point. Of course, it is linearly equivalent (over $\C$) to $x'=y,y'=-x$ with first integral $x^2+y^2$ which is the usual linear real center. 
The advantage to study Morse critical points over $\C$ is that we can use complex analysis and complex algebraic geometry. This is the point of view adopted in these notes.

The two equations (\ref{system}) and (\ref{ode}) are closely related. First, a polar change of variables transforms a plane system (\ref{system}) with a center to equivalent equation of the form (\ref{ode}) with a center along the interval $[0,2\pi]$. Second, if the family of functions $f(.,y)), x\in [0,1]$ is replaced by its Fourier series $\hat{f}(.,y)$ (so $\hat{f}(x+1,y) = \hat{f}(x,y)$ ) and the equation (\ref{ode}) has a center at $y=0$, then the new system
\begin{align}
\label{ode1}
\frac{dy}{dx} + \hat{f}(x,y)   = 0, (x,y)\in \R / \Z \times \R
\end{align}
will have all its orbits starting near the periodic solution $y=0$ on the cylinder $ \R / \Z \times \R$, periodic too. Of course, if $f$ is smooth, then the function $\hat{f}$ is only piece-wise smooth.
The transport map of (\ref{ode}) along $([0,1]$ becomes a return map for  (\ref{ode1}) and the definition of a limit cycle for (\ref{ode}) is straightforward too.
Actually, the scalar equation (\ref{ode}) in which $f$ is a regular function, should be considered as a simplified model of the eventually singular equation
$$
\frac{dy}{dx} = \frac{P(x,y)}{Q(x,y)} .
$$

We resume the above considerations in the following definitions, which make  sense both on $\R$ or  $\C$:
 \begin{definition}
 \label{def1}
 Let $ \varphi = \varphi(.; x_0,y_0)$ be the general solution of the equation  $dy + f(x,y)dx=0$   on the interval $ [x_0,x_1]$.
 \begin{description}
\item[ (i)] 
The solution $ \varphi = \varphi(.; x_0,y_0)$ is said to be periodic iff $\varphi(x_1; x_0,y_0) = x_0$
\item[(ii)]
The solution $ \varphi = \varphi(.; x_0,y_0)$ is said to be a limit cycle, provided that it is periodic and isolated, that is to say there is a neighbourhood of
its orbit on $S^1\times \R$ free of periodic solutions.
\item[(iii)]
the map $y \mapsto  \varphi(x_1; x_0,y) $ is the first return map of (\ref{ode}) in a neighbourhood of $(x_0,y=y_0)$. 
\item[(iv)]
The equation (\ref{ode}) defines a center in a neighbourhood of the periodic solution $\varphi$ provided that the first return map is the identity map in a neighbourhood of $y_0$.
If the return map is not the identity map, then we say that (\ref{ode}) defines a focus at the periodic solution $\varphi$.
 \end{description}
\end{definition}
 
 The center focus-problem for the equation (\ref{ode})  or  (\ref{system}) is, roughly speaking, to distinguish between a center and a focus. 
 The algebro-geometric content of the problem is as follows. Suppose, that (\ref{ode}) is polynomial, more precisely
 \begin{align}
\label{ode2}
\frac{dy}{dx} + \sum_{i=1}^n a_i(x) y^i , a_i \in \C[x], \deg a_i \leq n
\end{align}
so that $y=0$ is a periodic solution.
As we shall see in the next section, the first return map $y \mapsto \varphi(y)$
is analytic near $y=0$ and moreover
 $$
\varphi(y) = y + \sum_{n=1}^\infty c_n(a) y^{n+1} .
$$
where the coefficients $c_n=c_n(a)$ are polynomials  in the coefficients of $a_j = a_j(x)$, $j\leq n$. The condition $\varphi = id$ determines an infinite number of polynomial relations on the coefficients of $a_j$. By the Hilbert basis theorem, only a finite number of them are relevant, and they define an algebraic variety - the so called central variety $\mathcal C_n$ - in the vector space of all coefficients of the polynomials  $a_j$.  
The problem is therefore (as formulated by Lins Neto \cite{lins14} in the context of a polynomial foliation induced by  (\ref{system}))\\
\begin{center}{\it Describe the irreducible components of } $\mathcal C_n$.
\end{center}
\vspace{1cm}
 The content of the lectures is as follows.
 In section \ref{section2} we give a self-contained proof of an explicit formula, due to Brudnyi, for the solutions of the equation 
 $$
 \frac{dy}{dx} + \sum_{i=1}^n a_i(x) y^i = 0, 0\leq x \leq 1
 $$
in terms of iterated path integrals. 
 
 In section \ref{section3} we show how the perturbation theory of the Abel equation $ \frac{dy}{dx}  = a(x) y^2$ leads to the problem of vanishing of a suitable Abelian integral. The   conditions for vanishing of this Abelian integral give rise to a "moment problem" which has an elegant solution, due to Christopher. The solution is based on the well known L\"uroth  theorem. 
 
 Our main results concern the irreducible components of the center set of the scalar Abel equation and are formulated in section \ref{section4}.
 In section \ref{section41} we prove, that the set of scalar Abel equation with universal center (in the sense of Brudnuyi)  provide irreducible components of the Center set. 

 In sections \ref{section42},  \ref{section4b} we give a full description of the center set of two remarkable classical systems : quadratic vector fields and Li\'enard type equations. These results belong mainly to Dulac, Cherkas and Christopher, but we present them in the broader context of the present notes. In particular, the base field will be $\C$.
 All quadratic centers have a Darboux first integral, all Li\'enard centers come from "pull-back" (they are rationally reversible).
 
 It is fascinating that the centers of these two examples are of quite different nature, which motivates the so called Composition Conjecture for the Abel equation (\ref{ode}) to be discussed in section \ref{section43}.
 In this section we show that there are scalar Abel equations with a non-universal center (not related to a "pull back"). These equations have a Darboux type first integral, and among them we find the recent counter-example  to the Composition Conjecture mentioned above, found by Gin\'e, Grau and Santallusia \cite{ggs17}.

 \section{The first return map and the Brudnyi formula}
 \label{section2}
In this section we shall describe the return map of (\ref{ode2}) as a power series involving iterated path integrals. We give an explicit formula due to Brudnyi \cite{brud06}, which amounts to solve the differential equation. The classical approach to do this is by the Picard iteration method. If $y_0$ is the initial condition at $x_0$ of the differential equation 
$$dy = f(x,y) dx$$
 then the Picard iteration is
 $$
 y_{n+1}(x) = y_0 + \int_{x_0}^x y_n(t) dt 
 $$
 where $y_n$ tends to the solution of the equation as $n\to \infty$. 
 We illustrate this on the example $dy = y dx$. If $y_0$ is the initial condition at $x = 0$ then
  \begin{align*}
\label{exp}
y_1(x) & = y_0 + \int_0^x y_0 dt\\
y_2(x) & = y_0 + \int_0^x y_1(t) dt = y_0 +  \int_0^x y_0 dt + \iint_{0\leq t_2\leq t_1\leq x}  y_0 dt_1 dt_2
\end{align*}
 As 
 $$
\int \dots \int_{0 \leq t_n\leq \dots \leq t_1\leq x}y_0 dt_1   \dots d t_n = y_0 \frac{x^n}{n!} 
 $$
 we get $y(x) = y_0 e^x$ as expected. The multiple (or iterated) integrals above appear in a similar way  in the non-autonomous linear $dy=a(x) y dx$, or even non-linear case $dy=f(x,y) dx$.
 The non-linear case is more involved, it is reduced to the linear one, but after introducing  infinitely many new variables $y, y^2, y^3, \dots $. To get around this reduction we shall use a simple Ansatz, for which we 
 need a formal definition of iterated integral.

Let $Ass_\omega$ be the graded free associative algebra generated by the infinite dimensional vector space of differential one-forms $\omega = a(x,y) dx$, $a \in \C\{x,y\}$. Its elements are non-commutative polynomials in such one-forms.
The differential operator 
$$D :  Ass^1_\omega  \to Ass^1_\omega $$
$$
D( a(x,y)  dx) = \frac{\partial}{\partial y} a(x,y) dx
$$
induces a differential operator on  $ Ass_\omega$  which acts by  the Leibnitz rule. The readers familiar with the Picard-Lefschetz theory will regognize in $D$ an avatar of the covariant derivative of an Abelian integral along the level sets $\{y=const \}$. 

To save brackets, it is convenient to introduce the following notation
\begin{equation}
\label{Domega}
D\omega_1 \omega_2 \dots \omega_n = D(\omega_1 \omega_2 \dots \omega_n ) 
\end{equation}
so that (using brackets)
$$
D \omega_1 \omega_2 =D (\omega_1 \omega_2) =  (D \omega_1 ) \omega_2 +  \omega_1(D \omega_2) .
$$
and
$$
D \omega_1 D \omega_2 = D ( \omega_1 D \omega_2 )=  (D  \omega_1)( D \omega_2) + \omega_1 (D^2 \omega_2) .
$$
If we  use the notation
$$
D^k\omega = \omega ^{(k)} 
$$
then
$$
D \omega_1 \omega_2 = \omega_1' \omega_2+ \omega_1 \omega_2'
$$
and
$$
D \omega_1 D \omega_2 = ( \omega_1  \omega_2')' =  \omega_1'  \omega_2' +  \omega_1  \omega_2'' .
$$
For $\omega_1\omega_2 \dots \omega_n \in  Ass_\omega^n$, $\omega_k= \varphi_k(x,y) dx$,
define the iterated integral  $\int_{x_0}^x \omega_1\omega_2 \dots \omega_n$ of length $n$, as equal to
\begin{equation}
\label{multipleintegral}
 \int \dots \int_{x_0 \leq t_n\leq \dots \leq t_1\leq x} \varphi_1(t_1,y) \dots \varphi_n(t_n,y) dt_1   \dots d t_n .
\end{equation}
The iterated integral allows also a recursive definition (hence the name) :
\begin{equation}
\label{iteratedintegral}
\int_{x_0}^x \omega_n\omega_{n-1} \dots \omega_1 =  \int_{x_0}^x (\varphi_n (t) \int_{x_0}^t \omega_{n-1} \dots \omega_1)dt  
\end{equation}
where in the case $n=1$ we have the Riemann integral $\int_{x_0}^x \omega_{1}$.
We note, that the usual notation for the multiple integral (\ref{multipleintegral}) is $\int_{x_0}^x \omega_n\omega_{n-1} \dots \omega_1$ on the place of
 $\int_{x_0}^x \omega_1\omega_2 \dots \omega_n$, see  \cite[Chen]{chen77} or \cite[Hain]{hain87}.  The reason to prefer the definition (\ref{iteratedintegral}) is that it is 
 better adapted to applications in differential equation, e.g. \cite{gavr05}. Recall in this context, that 
 $$
 \int_{x_0}^x \omega_n\omega_{n-1} \dots \omega_1 = (-1)^n \int^{x_0}_x \omega_1\omega_2 \dots \omega_n .
 $$
For a short summary of properties of iterated integrals  which we use see  \cite[Appendix]{gavr05}, \cite[section 2]{gmn09}.

\begin{theorem}
\label{firstintegral} With the notation (\ref{Domega}),
a first integral of the differential equation $dy + f(x,y) dx = 0$  is given by the following recursively defined convergent  series
\begin{align}
\label{first2}
\boxed{
\varphi(x_0;x,y) = y + \int_{x_0}^x \omega +  \int_{x_0}^x \omega D \omega + \int_{x_0}^x \omega D \omega D \omega +
\dots  }
\end{align}
where
$$
\omega =  f(x,y) dx .
$$
The general solution of (\ref{ode}) with initial condition $(x_0,y_0)$ is given by $$y= \varphi(x;x_0,y_0) .$$
\end{theorem}
\begin{example}
In the linear case 
$$ y' + \alpha y = 0  \iff dy + \alpha y dx = 0
$$ we obtain
\begin{align*}
\varphi(x_0;x,y) = & y ( 1+  \alpha  \int_{x_0}^x dx + \alpha^2  \int_{x_0}^x dx.dx + \dots)\\
= & y( 1 + \alpha(x-x_0) + \alpha^2  \frac{(x-x_0)^2}{2} + \dots
 = & y e^{\alpha (x-x_0) }
\end{align*}
and the general solution is
\begin{align*}
y = \varphi(x;x_0,y_0) = & y_0 e^{\alpha (x_0-x) } .
\end{align*}
In the quadratic case  $$dy + 2 x y^2 dx = 0, \omega = 2 x y^2 dx $$
we compute recurisvely
\begin{align*}
 \int_{x_0}^x \omega &= \int_{x_0}^x  2 x y^2 dx = x^2-x_0^2 \\
 \int_{x_0}^x \omega D\omega & =  \int_{x_0}^x 2 x y^2 dx . 4  x y dx  = y^3 (x^2-x_0^2)^2 \\
  \int_{x_0}^x \omega D\omega \dots D\omega &= (x^2-x_0^2)^n
\end{align*}
Therefore we 
get the first integral
\begin{align*}
\varphi(x_0;x,y) 
 = & y + y^2 (x^2-x_0^2) + y^3  (x^2-x_0^2)^2 + \dots 
 \end{align*}
and the corresponding general solution is
\begin{align*}
y= \varphi(x ; x_0,y_0) & = y_0 + y_0^2 (x_0^2-x^2) + y^3  (x^2_0-x^2)^2 + \dots \\
& =  \frac{y_0}{1-y_0(x_0^2-x^2)} .
\end{align*}
\end{example}
\begin{proof}[Proof of Theorem \ref{firstintegral}.]
 We first verify, that for every fixed $x_0$, the function $\varphi(x_0;x,y)$ is a first integral :
\begin{align*}
d\varphi(x_0;x,y)&=\frac{ \partial}{\partial x} \varphi(x_0;x,y) dx + \frac{ \partial}{\partial y} \varphi(x_0;x,y) dy \\
&= \omega + \omega \int_{x_0}^x D \omega + \omega \int_{x_0}^x D \omega D \omega + \omega \int_{x_0}^x D \omega D \omega D \omega + \dots \\
 &+  (1 +    \int_{x_0}^x D \omega +  \int_{x_0}^x D \omega D \omega +  \int_{x_0}^x D \omega D \omega D \omega + \dots) dy \\
 &= ( \omega + dy)  \frac{ \partial}{\partial y} \varphi(x_0;x,y) dy = 0 .
\end{align*}
As $\varphi(x_0;x_0,y_0) = y_0$ then the level set $\{ (x,y) : \varphi(x_0;x,y) = y_0 \}$ contains both $(x_0,y_0)$ and $(x,y)$. By symmetry
$$
y= \varphi(x;x_0,y_0) 
$$
is the solution
 of (\ref{ode}) with initial condition $y(x_0)=y_0$. The convergency proof is by standard a priori estimates (omitted)
\end{proof}
Note that for fixed $x_0,x_1$ the two return  maps
$$
y \mapsto \varphi(x_1;x_0,y) , \;\; y \mapsto \varphi(x_0;x_1,y) 
$$
are mutually inverse. Therefore
$\varphi(x_1;x_0,.) = id$ if and only if $ \varphi(x_0;x_1,.) = id$. 
Using Theorem \ref{firstintegral} we can give explicit
center conditions.  Assume that
$$
f(x,y) = \sum_{i=1}^\infty a_i(x) y^{i+1} .
$$
and develop the return map $ \varphi(x_0;x_1,y) $  as a power series in $y$
\begin{equation}
\label{returnmap}
\varphi(x_0;x_1,y) = y + \sum_{n=1}^\infty c_n(a) y^{n+1} .
\end{equation}
If we denote, by abuse of notations, $a_{i}= a_{i}(x) dx $
then we get for the first few coefficients $c_n(a)$
\begin{align*}
c_1(a) & = \int_{x_0}^{x_1}  a_1 \\
c_2(a) & = \int_{x_0}^{x_1}  a_2 + 2 a_1a_1\\
c_3(a) & = \int_{x_0}^{x_1} a_3 + 2 a_2 a_1 + 3 a_1a_2 + 6 a_1^3 \\
c_4(a) & = \int_{x_0}^{x_1} 
a_4 + 2 a_3 a_1 + 3 a_2^2 + 4 a_1a_3 + 6a_2a_1^2 + 8 a_1a_2a_1 + 12 a_1^2a_2 + 24 a_1^4
\end{align*}
and so on. The general form of the coefficients $c_n(a)$  is found
immediately from Theorem \ref{firstintegral}. We resume this in the following

\begin{theorem}[Brudnyi's formula]
\label{brf}
The coefficients $c_n(a)$  of the first return map (\ref{returnmap}) for the differential equation
$$
\frac{dy}{dx} + \sum_{i=1}^\infty a_i(x) y^{i+1}   = 0, x\in [x_0,x_1]
$$
are given by the formulae
$$
c_n(a)  = \sum_{i_1+\dots + i_k=n} c_{i_1,\dots,i_k}  \int_{x_0}^{x_1} a_{i_1} \cdots a_{i_k} 
$$
where
\begin{align*}
c_{i_1} & = 1\\
c_{i_1,i_2} & = i_2+1 \\
c_{i_1,i_2,i_3} & = (i_3+1)(i_2+1)\\
\vdots & = \vdots \\
c_{i_1,\dots,i_k}  & = (i_k+1)(i_k+i_{k-1} +1) \dots (i_k + \cdots i_2 +1) 
\end{align*}
\end{theorem}
The above formula   was deduced first by Brudnyi  \cite[p.422]{brud06} under equivalent form, see also \cite[Proposition 2.4]{bry10} in the case (\ref{abel}).

\begin{corollary}
The equation (\ref{ode}) has a center on the interval $[x_0,x_1]$ if and only if $c_n(a)=0$, for all integer $ n \geq 1$.
\end{corollary}
\begin{example}
Suppose that the equation $$\frac{dy}{dx} + a_1(x)y^2 + a_2(x)y^3 + \dots = 0 $$ has a center on the interval $[x_0,x_1]$. Then, using as above  the  notation $a_{i}= a_{i}(x) dx $ we have
\begin{align*}
c_1&= \int_{x_0}^{x_1} a_1  = 0 \\
c_2& = \int_{x_0}^{x_1} a_2  + 2  \int_{x_0}^{x_1} a_1^2  = 0 .
\end{align*}
The identity $$2  \int_{x_0}^{x_1} a_1^2 =  (\int_{x_0}^{x_1} a_1)^2$$ implies then, that $ \int_{x_0}^{x_1} a_2  = 0$.
If we consider more specifically the Abel equation 
 \begin{equation}
 \label{abelex}
 \frac{dy}{dx} + a_1(x)y^2 + a_2(x)y^3  = 0 
 \end{equation}
 then taking into consideration that $\int_{x_0}^{x_1} a_1^3 = 0$ and 
 $$\int_{x_0}^{x_1}   a_2 a_1 + 3a_1a_2 = \int_{x_0}^{x_1}    a_1 a_2 = 0$$
  we obtain $c_3=  \int_{x_0}^{x_1} a_1 a_2 $. Therefore a necessary condition for the Abel equation (\ref{abelex}) to have a center on $[x_0,x_1]$ is
  \begin{equation}
  \label{nc}
  \int_{x_0}^{x_1} a_1  = 0,  \int_{x_0}^{x_1} a_2  = 0, \int_{x_0}^{x_1} a_1 a_2  = 0
  \end{equation}
 If we suppose that $a_1$, $a_2$ are polynomials of degree at most two, these conditions are also sufficient \cite{bfy99}. The case $\deg a_1, a_2 = 3$ can be studied similarly, see \cite{bfy00}.

\end{example}

In general, an obvious sufficient condition to have a center is therefore 
\begin{align}
\label{ucenter}
 \int_{x_0}^{x_1} a_{i_1} \cdots a_{i_k} = 0, \forall i_j, k \geq 1 .
\end{align}
Centers with the    property (\ref{ucenter}) were called {\it universal} in \cite{brud06}. 

Consider, more specifically, the following equation with polynomial coefficients $a_i$
\begin{align}
\label{an}
dy + \sum_{i=1}^n y^{i+1}  a_i(x) dx = 0,  a_i(x)  \in K[x] .
\end{align}

\begin{theorem}[Brudny, \cite{brud06}, Corollary 1.20]
The polynomial equation (\ref{an}) has an universal center  on the interval $[x_0,x_1]$, if and only if, it is a pull back of some polynomial equation 
\begin{align}
\label{bn}
dy = (\sum_{i=1}^n b_i(\xi) y^{i+1} ) d\xi,  b_i(\xi) \in K[\xi] .
\end{align}
via a suitable polynomial map $\xi = \xi (x)$ having the property $\xi(x_0)=\xi(x_1)$.
\end{theorem}
Not all centers of (\ref{an}) are universal, as discovered recently in \cite{ggs17}.

For a further use, note that an obvious consequence from $F(x,y)\equiv y$ is that
$$
\int_{x_0}^{x_1} a_1(x)dx = \int_{x_0}^{x_1} a_2(x)dx = 0.
$$
This will be used when studying the center problem for the so called degenerate Abel equation of first kind
$$
\frac{dy }{dx} = a_1(x)y^2 + a_2(x) y^3 .
$$

\section{Bifurcation functions and a Theorem of Christopher}
\label{section3}
In this section we study the following perturbed Abel differential equation on the interval $[0,1]$
$$
y' = a(x) y^2  - \sum_{j= 1}^\infty \varepsilon^j(y^2 p_j(x)+ y^3 q_j(x) )
$$
or equivalently
\begin{align}
\label{perturbed1}
\frac{dy}{y^2} = a(x)  dx - \varepsilon \omega_1 - \varepsilon^2\omega_2 - \dots 
\end{align}
where 
$$
\omega_j = (p_j(x)+ y q_j(x)) dx
$$
and $a=a(x),p_j=p_j(x), q_j=q_j(x)$
are polynomials of degree $$\deg a = n, \deg p_j \leq n, \deg q_j \leq n $$
and $\varepsilon$ is a small parameter.
For $\varepsilon = 0$  (\ref{perturbed1}) has a first integral
$$
H(x,y)= \frac1y + A(x), \; A(x) = \int a(x) dx .
$$

\emph{How many limit cycles has the perturbed system (\ref{perturbed1})  on the interval $[0,1]$? }

Recall from the preceding section that a solution $y(x)$ such that $y(0)=y(1)$) is called periodic on $[0,1]$.
A limit cycle of  (\ref{perturbed1}) on $[0,1]$ is therefore an isolated  periodic solution on $[0,1]$.

The number of the limit cycles in a compact set are bounded by the number of the  zeros of the so called \emph{bifurcation function}, which 
we define bellow. A limit cycle which remains bounded when $\varepsilon \to 0$, tends to a periodic solution of the non perturbed system.
If the non-perturbed system ($\varepsilon = 0$) has a periodic solution, then necessarily $A(0)=A(1)$, which already implies that it
has a center.  
For this reason we assume from now on that 
$A(0)=A(1)=0$, so that 
\begin{equation*}
dy = a(x) y^2 dx \Leftrightarrow dH = 0
\end{equation*}
 has a center along $0\leq x \leq 1$.
The perturbed equation
can be written
\begin{align}
\label{perturbed}
d H - \varepsilon \omega_1 - \varepsilon^2\omega_2 - \dots = 0 .
\end{align}
For a solution $y(x)$, let $\mathcal P_\varepsilon$ be the first return map   which sends the initial condition $y_0=y(0)$ to $y_1=y(1)$. We parameterise $\mathcal P_\varepsilon$ by $h=\frac1y = H(0,y)=H(1,y)$ and note that $\mathcal P_\varepsilon$ is analytic both in $h$ and $\varepsilon$ (close to zero). We have therefore for the first return map
\begin{align}
\label{return1}
\mathcal P_\epsilon (h)- h= \varepsilon^k M_k(h) + O(\varepsilon^{k+1}) , \; M_k \neq 0
\end{align}
The  function $M_k$ is the\emph{ bifurcation function}, associated to the equation (\ref{perturbed1}). It is also known as "first non-zero Melnikov function".
The reader may compare this  to (\ref{first2}) which is another representation of the first return map, defined for small $y$. As we shall see, the bifurcation function is globally defined. Therefore 
for every compact set $K$, $[0,1] \subset K \subset \R^2$ and all sufficiently small $|\varepsilon|$, the number of the limit cycles of (\ref{perturbed1}) in $K$ is bounded by the number of the zeros of the bifurcation function $M_k$ (counted with multiplicity).

$M_k$ allows  an integral representation  
$$
M_k(h) = \int_{\{H=h\}} \Omega_k
$$
where the integration is along the level set 
$$\{H=h\} = \{(x,y) : 1/y + A(x)= h, 0\leq x \leq 1 \} .$$
 The differential form $\Omega_k$ is
computed by the classical Fran\c{c}oise's recursion formula  \cite{fran96,ilie98a,rous98} as follows:
\begin{quotation} \em
If $k=1$ then $\Omega_1 = \omega_1$, otherwise
\begin{align}
\Omega_m = \omega_m + \sum_{i+j=m} r_i \omega_j, \; 2 \leq m \leq k
\end{align}
and the functions $r_i$, $1\leq i \leq k-1$ are determined successively from the identities $\Omega_i = d R_i + r_i dH$.

\end{quotation}

The first order Melnikov function $M_1$   is computed   in \cite{bfy00}, and probably earlier,  e.g. Lins Neto \cite[section 3]{lins80}.
We have
\begin{align*}
M_1(h) &=   \int_{\{H=h\}} \omega_1 \\ &= \int_{\{H=h\}} p_1(x) dx + yq_1(x) dx\\
& = \int_0^1 p_1(x) dx + \int_0^1 \frac{q_1(x)}{h - A(x)}  dx  \\
& = \int_0^1 p_1(x) dx + \sum_{k=0}^\infty h^{-k-1} \int_0^1 q_1(x)  A^k(x)dx .
\end{align*}
$M_1$ vanishes identically if and only if
$$
\int_0^1 p_1(x) dx = 0, \; \int_0^1 q_1(x) A^k(x) dx = 0, \; k= 0,1,2, \dots
$$
which is the content of the polynomial moment problem for $q_1$ and $A$.  If $M_1=0$ by the above formula we get
\begin{align*}
M_2(h)= \int_{\{H=h\}} r_1 \omega_1 + \int_{\{H=h\}} \omega_2 
\end{align*}
where $r_1$ is computed from the identity $\omega_1= d R_1 + r_1 dH$. As $d \omega_1 = d r_1 \wedge dH$ then $d r_1 = \omega_1'= \frac{d\omega_1}{d H}$ is the Gelfand-Leray form of $\omega_1$.
From the identity $H(x,y(x,h)) \equiv h$ we have $ \frac{\partial  y}{\partial h} = - y^2$ and hence
$$
r(x,y)= \int_0^x \omega_1' =  -  \int_0^x y^2 q(x)  .
$$
We conclude
\begin{proposition}[formula (2.8) in \cite{gavr05}]
Under the hypothesis $M_1=0$ the second Melnikov function reads
\begin{align}
M_2(h) &= \int_{\{H=h\}} \omega_1 \omega_1' +  \int_{\{H=h\}} \omega_2 
\end{align}
where
$$
\omega_1=  p_1(x) dx + yq_1(x),\; \omega_1'= -y^2 q_1(x) dx,\; \omega_2= p_2(x) dx + yq_1(x) .
$$
\end{proposition}

The hypothesis $M_1=0$ is of   interest for us, as it will allow to compute the tangent space to the center  set at the point $(a,0)$, see the next section. In full generality this vanishing problem is solved by  Pakovich and Muzychuk \cite{pamu09}. For our purposes however, the polynomial $a(x)$ can be taken in a general position, as in the following  
\begin{theorem}[Christopher \cite{chri00}]
\label{ccth}
Assume that $a(0)\neq 0$ and $a(1)\neq 0$. 
The multivalued transcendental function
$$
I(h) =  \int_0^1 \frac{q(x)}{h - A(x)}  dx  
 $$
vanishes identically,   if and only if the polynomials $Q = \int q$ and $A$ satisfy the following "Polynomial Composition Condition" (PCC) : 
\begin{center}
There exist polynomials $\tilde Q, \tilde A, W$, such that
$$A= \tilde A \circ W, Q= \tilde Q \circ W  , W(0)=W(1) .
$$
\end{center}
\end{theorem}
Before recalling  the elegant proof of Christopher, we put $I$ in the broader context of the Picard-Lefschetz theory.

The  function $I(h)$ is well defined for sufficiently big $h$, and has an analytic continuation in a complex domain to certain multivalued function.
It is in fact  an Abelian integral depending on a parameter. 
 More precisely, consider the genus zero affine curve
$$
\Gamma_h= \{(x,y)\in \C^2: \frac 1y +A(x) = h \} .
$$
It is a Riemann sphere with $n+2$ removed points, provided that $h \neq 0$. The removed points correspond to $(x=x_i(h), y=0)$, where $A(x_i(h))\equiv h$, and to $(x=\infty,y=0)$.
Given a divisor $m=P_0+P_1$ on $\Gamma_h$, where 
$$P_0 = ( 0,h), \;P_1=(1,h) $$
we define 
a singular algebraic curve $\Gamma_h'$. As a topological space it is just the curve $\Gamma_h$ with the two points $P_0$ and $P_1$ identified to a point $m$. The structural sheaf of $\Gamma_h'$ is the same as the structural sheaf of $\Gamma_h$, except at the point $m\in \Gamma_h'$. At this point a function $f$ is said to be regular, if it is is regular on $\Gamma_h$, and moreover $f(P_0)=f(P_1)$. 
The path $[0,1]$ connecting the points $x=0$ and $x=1$   closes on the singular algebraic curve $\Gamma_h'$. 
The function $I(h)$ is an Abelian integral on $\Gamma_h'$.
We note that the above procedure is easily generalized to arbitrary divisor $m$ on $\Gamma_h$, which fits the generalized moment problem, as defined in \cite[Conjecture 1.7]{bfy99}.

The homology group $H_1(\Gamma_h', \Z)$ is of dimension $n+2$.
 It is generated by $n+1$ simple closed loops $\gamma_i= \gamma_i(h)$ which make one turn around the $n+1$ punctures $x_i(h)$ on $\Gamma_h$,  $A(x_i(h))-h=0$, as well the loop connecting $0$ and 
 $1$ on the singularized curve $ \Gamma_h'$. The monodromy of the loop $[0,1]$ is shown on the figure. It follows that the orbit of $[0,1]\in H_1(\Gamma_h', \Z) $ under the action of the fundamental group of 
 $\C\setminus \Delta$ contains the  $\gamma_i - \gamma_j$, where $A(x_i(0))=A(x_j(0)) = 0$.
The Abelian integral $I(h)$ on the Riemann sphere $\Gamma_h$, can be  presented as a zero-dimensional Abelian integral as follows
\begin{align*}
\int_{\gamma_i(h) - \gamma_j(h)}   yq(x) dx &= \int_{\gamma_i(h) - \gamma_j(h)}  \frac{q(x)}{h-A(x)} dx \\
&= -2\pi i (\frac{q(x_i(h))}{A'(x_i(h))} - \frac{q(x_j(h))}{A'(x_j(h))} )\\
&= -2\pi i \frac{d}{dh} [ \frac{q}{A'}(x_i(h)) -  \frac{q}{A'}(x_j(h)) \\
&= -2\pi i \frac{d}{dh} [ Q(x_i(h)) -  Q(x_j(h))]
\end{align*}
where
$$
Q(x) = \int q(x)dx
$$
is a primitive of $q$, 
and  $x_i(h)$ are the roots of the polynomial $A(x)-h$ (therefore $A'(x_i(h)) . x_i'(h) \equiv 1$).

We denote, 
\begin{align}
J(h) = \int_{x_i(h)-x_j(h)} Q = Q(x_i(h)) -  Q(x_j(h))
\end{align}
and call $J$ an Abelian integral of dimension zero along the zero-cycle 
$$x_i(h)-x_j(h) \in H_0(\{A(x)=h \}, \Z) $$
(\cite{gamo07}, Definition 1 ).
If the Abelian integral $I(h)$ vanishes identically, then the same holds true for $J'(h)$, hence $J(h)= const.$ and it is easy to check that the constant is zero,  $J(h) \equiv 0$. 
The set of rational functions $Q$ such that $ Q(x_i(h)) \equiv Q(x_j(h)$ is a subfield of the field of all rational functions $\C(x)$. By the L\"uroth theorem this subfield is of the form $\C(W)$ for suitable rational function $W$. It follows that $Q= \tilde Q \circ W, A = \tilde A \circ W$ and it is easily seen in this case that $W$ is a polynomial. We may also assume that $W(x_i(h))\equiv W(x_j(h)) $,  as in the opposite case we may 
replace the variable $x$ by $W$ and reason by induction on the polynomials $Q= \tilde Q \circ W,$ and $ A = \tilde A \circ W $.

We have proved that if the Abelian integral $I(h)$ vanishes identically, then the polynomials $Q$ and $A$ satisfy the PCC, and the opposite claim is obvious. This completes the proof of Theorem \ref{ccth}
\endproof

\section{Irreducible components of the Center set}
\label{section4}

An affine algebraic variety $V$ in $\C^n$ is the common zero locus of a finite collection of polynomials $f_i \in \C[z_1,\dots ,z_n]$. The variety $V$ is said to be irreducible, if for any pair of closed varieties $V_1,V_2$ such that $V= V_1\cup V_2$, either $V_1=V$ or $V_2=V$. Of course, it might happen that a variety $V$ is reducible $V= V_1\cup V_2$, where $V_1,V_2\neq V$. In this case we may ask whether $V_1$ and $V_2$ are further reducible and so on. It is a basic fact of commutative algebra that in this way only a finitely many irreducible subvarieties $V_i\subset V$ can be found, and more precisely

\emph{Any variety $V$ can be uniquely expressed as a finite union of irreducible varieties $V_i$ with $V_i \subsetneqq V_j$ for $i\neq j$. \cite{harr95}}

The  varieties $V_i$ which appear in the finite decomposition $$V = \cup_i V_i$$ are the\emph{ irreducible components} of $V$. 

Let $W \subset V$ be another algebraic variety. Is $W$ an irreducible component of $V$? It is usually easy to verify, whether $W$ is irreducible. It is much harder to check that $W$ is an irreducible component of $V$. Indeed, it might happen that $W \subsetneqq V_i$ where $V_i$ is an irreducible component of $V$. To verify this, one may compare the dimensions of the tangent spaces $T_x W$ and $T_x V$ at some smooth point $x\in V\cap W$ (one point $x$ is enough!).
Then \emph{  $W \subsetneqq V_i$ if and only if $T_x W  \subsetneqq T_x V$ }. Of course, there might be no way to know that $x$ is a smooth point, in which case  we   use  the tangent cones $TC_x W$ and $TC_x V$. For every $x\in W$ on an irreducible variety $W$ holds $\dim TC_x W = \dim W$. Thus, for irreducible varieties 
$W \subset V$ holds 
$$
\dim TC_x W  < \dim TC_x \Leftrightarrow W  \subsetneqq V .
$$
The choice of $x \in W$ is irrelevant, which allows a great flexibility. 

The above approach will be applied in the case when $V$ si the center set of the equation (\ref{ode}), and $W$ is a subset of equations with a center.
In the planar case (\ref{system}) this approach was developped by Movasati \cite{mova04}. He observed that the vanishing of the first Melnikov function, related to one-parameter deformations (arcs) of systems (\ref{system}) with a center,  provides equations for the tangent cspace $T_x W$, while the vanishing of the second Melnikov function provides equations for the tangent cone $TC_x W$.
This remarkable connection between algebraic geometry and dynamics will allow us to go farther in the description of irreducible components of the center set. 
We adapt the approach of Movasati \cite{mova04}  and Zare \cite{zare17} to (\ref{ode}) in the context of 
 the set $\mathcal A_n$ of Abel differential equations
\begin{align}
\label{abel2}
\frac{dy}{dx} = a(x) y^2+ b(x)y^3
\end{align}
parameterised by the polynomials $a(x), b(x)$ of degree at most $n$. They form therefore a vector space of dimension $2n+2$, and consider the subset $\mathcal C_n\subset \mathcal A_n$ of Abel differential equations having a  center on the interval $0\leq x\leq 1$. As we saw in the preceding section, $\mathcal C_n$ is defined by finitely many polynomial relations $c_n(a,b) = 0$ and therefore is an algebraic set.

\subsection{Universal centers define irreducible components of the center set}
\label{section41}

If the integer $k>1$ divides $n+1$, then we denote by $\mathcal U_{n/k}\subset \mathcal C_n\subset \mathcal A_n$ the algebraic closure of the set of pairs of polynomials $(a,b)$
(or   Abel equations (\ref{abel2})), such that the following Polynomial Composition Condition (PCC) is satisfied
 
\emph{There exist polynomials $\tilde A, \tilde B, W$ of degrees  $(n+1)/k$, $(n+1)/k$, $k$, such that}
\begin{equation} \tag{PCC} A= \tilde A \circ W, B= \tilde B \circ W  , W(0)=W(1) .
\end{equation}

 The differential form associated to (\ref{abel2})
 $$
dy - (a(x) y^2+ b(x)y^3) dx = dy - y^2 dA(x) - y^3 dB(x)
 $$
 is a pull back of the differential form
\begin{equation}
\label{pback}
 dy- (\tilde{A}'(w) y^2 + \tilde{B}'(w) y^3) dw = dy - y^2 d \tilde{A}(w) - y^3 d \tilde{B}(w)
 \end{equation}
 under the map $(x,y) \to (w,y)$, where $w=W(x)$. In other words the equation (\ref{abel2}) is obtained from
 $$
 \frac{dy}{dw} = \tilde{A}'(w) y^2 + \tilde{B}'(w) y^3
 $$
 via the substitution $w=W(x)$. 
 This, combined to $W(0)=W(1)$ implies that the set of Abel equations $\mathcal U_{n/k}$ have a center at $y=0$ along $[0,1]$.  Of course one could check directly that the center conditions $c_n(a)=0$  are satisfied for all $n$ (Theorem \ref{brf}). Indeed, the iterated integrals  
 $
  \int_{x_0}^{x_1} a_{i_1} \cdots a_{i_k} 
  $
  vanish, because they are pull backs under $W$ of iterated integrals along an interval, contractible to the point 
 $W(x_0)=W(x_1)$. 
  Following Brudnuyi \cite{brud06}, we say that  (\ref{abel2}) determines an \emph{universal center} if and only if
   $$
  \int_{x_0}^{x_1} a_{i_1} \cdots a_{i_k} = 0, \forall i_j \in \N .  
  $$
  
 \emph{ It is shown then that a center is universal, if and only if the corresponding equation  (\ref{abel2})   is a pull back under an appropriate polynomial as above, see Brudnyi \cite[Corollary 1.20]{brud06}.
}  Thus, the universal centers are exactly those, obtained by a polynomial pull back in the sense  (\ref{pback}), see the Polynomial Composition Condition  (PPC).

 Note that the universal center set $\mathcal U_{n/k}$ is an irreducible algebraic variety, as a Zariski open subset of it is  parametrized by the polynomials $\tilde A, \tilde B, W$ of degrees respectively $(n+1)/k, (n+1)/k, k$. The main result of the section is
 \begin{theorem}
 \label{pullback}
 The  algebraic sets  $\mathcal U_{n/k}$ are irreducible components of the center set $\mathcal C_n$ of the Abel equation
$$
\frac{dy}{dx} = a(x) y^2+ b(x)y^3, \deg a, \deg b \leq n .
$$
 \end{theorem}
 
 We shall illustrate first the idea of the proof of Theorem \ref{pullback} on the rather elementary case $k=n$. The closure of the universal center set $ \mathcal U_{n/n+1} $ consists of Abel equations (\ref{abel2}) such that
 $$
 \deg a, \; \deg b \leq n, \; \int_0^1 a(x)dx = \int_0^1 b(x)dx = 0 
 $$
 and moreover the polynomials $a(x), b(x)$ are co-linear. Thus, $ \mathcal U_{n/n+1} $ is identified to the  vector space of pairs of polynomials $(a(x),b(x))$ with the above properties, and is therefore of dimension $n+1$. 
Consider now the point $(a(x),0) \in U_{n/n+1} $ where $a(x)$ is a degree $n$ polynomial. 
\begin{proposition}
The tangent space $T_{(a,0)} \mathcal U_{n/n+1} $ is a vector space of dimension $n+1$,
which
consists of pairs of
polynomials $(p,q)$ of degree at most $n$, such that
$q$ and $a$ are co-linear polynomials,  and $\int_0^1 p(x)dx = 0$
\end{proposition}
The proof is left to the reader.
Next, we compute the tangent cone $TC_{(a,0)} \mathcal C_n $ at $(a,0)$ to the center set $\mathcal C_n$. To avoid complications, we choose $a$ to be a non-composite polynomial.  
\begin{proposition}
\label{tangentc1}
Lat $a$ be a non-composite polynomial of degree $n$, such that $a(0)\neq 0, a(1)\neq 0$. Then
$$
TC_{(a,0)} \mathcal C_n = T_{(a,0)} \mathcal U_{n/n+1}
$$
\end{proposition}
The above  implies that
\emph{the algebraic set  $\mathcal U_{n/n+1} $ is an irreducible component of the center set   $\mathcal C_n$.}\\
{\em Proof of Proposition \ref{tangentc1}.}
Consider a one-parameter deformation 
\begin{align}
\varepsilon \to (a- \varepsilon p  + \dots, - \varepsilon q  + \dots) 
\end{align}
 of  (\ref{abel2}) at the point $(a,0)$.  For $\varepsilon=0$ the equation is
 $$
 \frac{dy}{y^2} = a(x) dx
 $$
 and has a first integral 
 $H(x,y)= \frac1y + A(x)$
 where $A$ is a primitive of $a$, $A(0)=A(1)$. The perturbed equation is
 $$
 dH - \varepsilon (p(x) + yq(x) dx + \dots = 0 .
 $$
 For sufficiently small $y, \varepsilon$ instead of $y$ we can use $h = H(x,0)$ as a variable and write for the return map $\varphi_\varepsilon$
  $$
 \varphi_\varepsilon (h) = h + \varepsilon M_1(h) + \dots
 $$
 The Melnikov function $M_1$, according to section \ref{section3}, is computed to be
$$
M_1(h) = \int_{H=h} p(x) dx + yq(x) dx = \int_0^1 p(x) dx + \int_0^1 \frac{q(x)}{h - A(x)}  dx  .
$$
Assuming that  for all sufficiently small $\varepsilon$ the deformed Abel equation belongs to the center set $\mathcal C_n$, implies $M_1=0$, which on its turn imposes rather severe conditions on the polynomials $p, q$.
First, $\int_0^1 p(x) dx = 0$ as follows already from (\ref{nc}). The second condition
$$
 \int_0^1 \frac{q(x)}{h - A(x)}  dx \equiv 0
 $$
 is well studied in a number of articles, and is known as the \emph{polynomial moment problem}, e.g. \cite{bry10} and the references there. For the case of a general $A$, see the Addendum by Pakovich in \cite{yomd03}. As $a(0)\neq 0, a(1)\neq 0$ then by Theorem \ref{ccth} we have that
  $\int_0^1 \frac{q(x)}{h - A(x)}  dx \equiv 0$ if and only if the composition condition holds true. As $A$ is supposed to be prime, this means that $A$ and $Q=\int q$ are co-linear polynomials. This completes the proof of Proposition \ref{tangentc1} and   Theorem \ref{pullback} in the case $k=n$.

 Note that in full generality, a vector $(p,q) $ which belongs to the tangent cone is a vector, such that there is a  one-parameter deformation 
$$\varepsilon \to (a+ \varepsilon^k p  + \dots, \varepsilon^k q  + \dots) $$
  at the point $(a,0)$ which belongs to the center set $\mathcal C_n$. The same arguments give the same constraints to the vector $(p,q)$. 
 
 \proof[Proof of Theorem \ref{pullback} in the general case]
Assume that the integer $k>1$ divides $n+1$ and consider the algebraic set $\mathcal U_{n/k}$ of Abel differential equations,  at $y=0$ along  $[0,1]$.
The proof follows the same lines as the case $k=n$, with the notable difference that  the second Melnikov function $M_2$ will be needed.

We compute first the tangent space to $\mathcal U_{n/k}$ at a general point $(a,0)$. 
Consider for this purpose the one-parameter deformation
\begin{align}
\label{abelfoliationeps}
\mathcal F_\varepsilon : \frac{dy}{y^2} = a(x)   dx - \varepsilon \omega_1 - \varepsilon \omega_2 - \dots
\end{align}
where
$$
\omega_i= p_i(x) dx  + y q_i(x) dx 
$$
are polynomial one-forms, $\deg p_i \leq n$, $\deg q_i \leq n$. As before we denote
$$
A = \int a ,\; P_i = \int p_i ,\; Q_i = \int q_i $$
where
$$ A(x)= \tilde{A}(W(x)), W(0)=W(1), P_i(0)=P_i(1), Q_i(0)=Q_i(1) .
$$

 The point $(a,0)$ belongs to $\mathcal U_{n/k} $ if and only if $A= \tilde{A} \circ W$ for some degree $k$ polynomial $W$.
\begin{proposition}
\label{cone}
The tangent space $T_{(a,0) }\mathcal U_{n/k} $ is the vector space of polynomials $(p_1,q_1)$
 such that
$$
P_1(x) = \tilde P_1 \circ W(x)  + R(x).  \tilde A' ( W(x)) , \; Q_1(x) =  \tilde Q_1 (W(x))  
$$
where $\tilde P_1 , \tilde Q_1  $ are arbitrary polynomials of degree at most $(n+1)/k$   and $R= R(x)$ is any degree $k$ polynomial, such that $R(0)=R(1)$.
\end{proposition}
The proof is straightforward, it suffices to consider the first order approximation  in $\varepsilon$ of the general deformation
$$
\omega_1^\varepsilon = d[(\tilde{A} + \varepsilon \tilde P) \circ (W+ \varepsilon R)(x)] + \varepsilon y d [\tilde Q \circ (W+ \varepsilon R) (x) ]
$$
of
$
\omega_1^0= a dx .
$

Next, we study the tangent cone $TC_{(a,0)} \mathcal C_n$. We need to compare the  affine varieties $T_{(a,0) }\mathcal U_{n/k} \subset TC_{(a,0)} \mathcal C_n$.
\begin{proposition}
In a sufficiently small neighbourhood of every general point $(p,q)\in T_{(a,0) }\mathcal U_{n/k}$ the tangent cones $ TC_{(a,0)} \mathcal C_n$ and  $T_{(a,0) }\mathcal U_{n/k}$ coincide.
\end{proposition}
The above Proposition shows that there is no irreducible component of $ TC_{(a,0)} \mathcal C_n$ which contains an irreducible component of $T_{(a,0) }\mathcal U_{n/k}$ of strictly smaller dimension.
This would imply Theorem \ref{pullback}.

The first Melnikov function, as in the case $k=n$, is $M_1= \int_0^1 p_1dx + yq_1 dy$. By Christopher's theorem $M_1=0$ implies that $q_1$ satisfies the composition condition
$$
Q_1(x)= \tilde Q_1 (W(x)) .
$$
Additional obstructions on the form of $p_1$ will be found by inspecting the second Melnikov function $M_2$.
Under the condition that $M_1=0$ we find  \cite[formula (2.8)]{gavr05}
$$
M_2 = \int_0^1 \omega_1 \omega'_1 + \int_0^1 \omega_2
$$
where the derivative $'$ is with respect to the parameter $h$. The identity $h= A(x)+ \frac1y$ shows that $y'=-y^2$ and $\omega_1'= - y^2 dx$, it is clearly a covariant derivative in a cohomology bundle (although we do not need this interpretation here). Therefore, for the iterated integral of leght two we find
\begin{align*}
 \int_0^1 \omega \omega'  & = - \int_{\{H=h\}} (p_1 dx+q_1 y dx)( y^2 q_1 dx) \\
 & = - \int_{\{H=h\}} (p_1 dx)( y^2 q_1 dx) \\
 &= \int_{\{H=h\}} y^2 q_1 P_1 dx
\end{align*}
where $P_1$ is a primitive of $p_1$. Indeed, $M_1=0$ implies the composition condition for $Q_1=\int q_1$ and $A$, that is to say the integral $\int_{\{H=h\}} y q_1 dx$ vanishes as a pull back. The same then holds true for its derivative $\int_{\{H=h\}} y^2 q_1 dx$ as well for the iterated integral $\int_{\{H=h\}} (y q_1 dx))( y^2 q_1 dx) $ . Further, the shuffle relation for iterated integrals
\begin{align*}
 \int_{\{H=h\}} (p_1 dx)( y^2 q_1 dx) &+  \int_{\{H=h\}} ( y^2 q_1 dx) (p_1 dx) \\ &= \int_{\{H=h\}} p_1 dx \int_{\{H=h\}} y^2 q_1 dx = 0
\end{align*}
Further, for $\int_0^1 \omega_2$ we find
\begin{align*}
 \int_0^1 \omega_2 =  \int_0^1 (p_2+  y q_2)   dx &= \int_0^1 \frac{d Q_2}{h-A(x)} \\
 & = - \int_0^1 \frac{Q_2 dA}{(h-A)^2}  + \frac{Q_2}{h-A} |_0^1 \\
 &  = -\int_0^1 y^2 Q_2 a dx .
\end{align*}
so that under the condition $M_1=0$ implies
$$
M_2(h)=  \int_{\{H=h\}} y^2 q_1 P_1 dx -  y^2 Q_2 a dx =  \int_{0}^1 \frac{q_1(x) P_1(x) - Q_2(x) a(x)}{(h-A(x))^2} dx .
$$
We apply Christopher's theorem to $M_2$ and conclude that the primitive of the polynomial $q_1(x) P_1(x) - Q_2(x) a(x)$ is a composite polynomial, it can be expressed as a polynomial function in $W(x)$, and therefore
$$
q_1(x) P_1(x) - Q_2(x) a(x) = P(W(x)) W'(x) 
$$
or equivalently
$$
\tilde{Q}_1'(W(x)) P_1(x) - Q_2(x) \tilde{A}'(W(x)) = R_1(W(x))
$$
for certain polynomial $R_1$. Assuming that $\tilde{Q}_1'$ and $\tilde{A}'$ are mutually prime, there exist polynomials $R_2, R_3$ such that
$$
\tilde{Q}_1'(W) R_2(W) - \tilde{A}'(W) R_3(W) = R_1(W)
$$
so
$$
\tilde{Q}_1'(W(x))  (P_1(x) -R_2(W(x)))- (Q_2(x)- R_3(W(x)))  \tilde{A}'(W(x)) = 0 .
$$
This implies finally that $ \tilde{A}'(W(x))$ divides $P_1(x) -R_2(W(x))$ and
$$
P_1(x) = R_2(W(x)) + R(x) \tilde{A}'(W(x)) .
$$
Proposition \ref{cone}, and hence Theorem \ref{pullback} is proved.
 \endproof

\subsection{An example : the central set of plane quadratic vector fields}
\label{section42}
Let $ \mathcal A_n$  be here  the set of all polynomial vector fields of degree at most $n$.
The only (non-trivial) case in which the center set $\mathcal C_n \subset \mathcal A_n$  is completely known is the quadratic one, $n=2$.
For comprehensive description and historical comments concerning the center-focus problem in the quadratic case see
 Zoladek \cite{zola94}.
To the plane quadratic vector field (\ref{system})
we associate a foliation $\mathcal F_\omega = \{\omega = 0\} $ on $\C^2$, defined by  the polynomial one-form
$$
\omega = P(x,y) dy - Q(x,y) dx   .
$$
The leaves of the foliation  are the orbits of the plane vector field (\ref{system}), and the restriction of the 
one-form $\omega$ on the leaves of $\mathcal F_\omega$  vanishes identically. 

In this section we assume that the polynomials $P,Q$ are of degree at most two, and
 the system has a center. 
As the foliation is over $\C$ we must be more careful in the definition. We shall say that a singular point is a center, if the point is non-generate, and has a local holomorphic first integral with a Morse critical point. 
Thus, i a neighbourhood of such a point, and up to a complex affine change of the variables, the system can be written in the form
\begin{align*}
\dot{x} &=x + P_2(x,y), \;
\dot{y}= - y + Q_2(x,y)
\end{align*}
for some homogeneous polynomials $P_2, Q_2$.  The following classical result is   implicit in Zoladek \cite[Theorem 1]{zola94},  and explicit in Lins Neto \cite[Theorem 1.1]{lins14}.
\begin{theorem}
The center set $\mathcal C_2$ of plane polynomial quadratic systems with a Morse center has four irreducible components. 
\end{theorem}
The result follows  essentially from the Dulac's classification \cite{dul08} of such Morse centers in a complex domain.
We sketch a proof. 
\proof
The foliation  $\mathcal F_\omega$, respectively the vector field (\ref{system}), is said to be logarithmic, if  
\begin{equation}
\label{log}
P(x,y) dy - Q(x,y) dx = f_1 \dots f_k \sum_{i=1}^k \lambda_i \frac{df_i}{f_i}, f_i \in K[x,y], \lambda_i \in K 
\end{equation}
for suitable polynomials $f_i$ and exponents $\lambda_i$.
As
$$
\sum_{i=1}^k \lambda_i \frac{df_i}{f_i} = d \log \Pi_{i=1}^k  f_i^{\lambda_i}
$$
then the logarithmic foliation  $\mathcal F_\omega$ has a first integral of Darboux type
$$
 \Pi_{i=1}^k  f_i^{\lambda_i} .
$$
Let $\mathcal L (d_1,d_2, \dots,d_k)$ denotes   the set of such logarithmic foliations (or plane vector fields) with
$$
\deg f_1 \leq d_1, \deg f_2 \leq d_2 , \dots, \deg f_k \leq d_k .
$$
For generic polynomials $f_i$ of degree $d_i$ the degree of the associated vector field is $\sum d_i-1$.
Therefore $\mathcal L (d_1,d_2, \dots,d_k)$  is quadratic, provided that
 $d_1=3$ or $d_1=1, d_2=2$ or $d_1=d_2=d_3=1$. This defines three large irreducible sets  of quadratic systems with a Morse center, 
$\mathcal L(3), \mathcal L(1,2), \mathcal L(1,1)$ respectively. The irreducibility of these algebraic sets follows from the fact that they are parameterised by the coefficients of the polynomials $f_i$ and by the exponents $\lambda_i$. We have one more exceptional irreducible set of systems with a Morse center which is
$$
\mathcal Q_4 = \mathcal L(2,3) \cup \mathcal A_2 .
$$
Here $\mathcal L(2,3)$ is the set of polynomial foliations as above, with a first integral $f_2^3/f_3^2$ where $\deg f_2 = 2, \deg f_3 = 3$. Generically such a foliation is of degree four, but it happens that its intersection $\mathcal Q_4$ with the space $\mathcal A_2$ of quadratic foliations is non empty and it is an irreducible algebraic set. The notation $\mathcal Q_4$ is introduced by Zoladek \cite{zola94}, the index $4$ indicates the co-dimension of the set in $\mathcal A_2$. To complete the proof we carefully investigate all cases of the Diulac's classification as reproduced in \cite[Theorem  7]{cerlin96}. It is seen there that in all ten cases found by Dulac, the corresponding quadratic system with a Morse center is either in one of the four cases above, or in their closures
$$
\overline{\mathcal{L}(3)}, \overline{\mathcal{L}(2,1)}, \overline{\mathcal{L}(1,1,1)}, \overline{ \mathcal{L}(3,2)}\cap \mathcal A_2 .
$$
To illustrate the last claim, consider a
quadratic foliation defined by the closed one form
$$
\omega_0= p_1p_2 \cdot \eta_0,\;  \eta_0 = \lambda_1 \frac{dp_1}{p_1}+  \lambda_2 \frac{dp_2}{p_2} + d q
$$
where $p_1,p_2,q$ are linear functions. The system has a   first integral   $p_1^{\lambda_1}p_2^{\lambda_2} e^q$,
and hence generically a Morse center, although it 
does not belong to any $\mathcal L (d_1,d_2, \dots,d_k)$.

However, the one-parameter family of one-forms
$$
\omega_\varepsilon = p_1p_2 (1+ \varepsilon q)( \lambda_1 \frac{dp_1}{p_1}+  \lambda_2 \frac{dp_2}{p_2} +\frac{1}{\varepsilon} \frac {d(1+ \varepsilon q)}{1+ \varepsilon q}) \in \mathcal{L}(1,1,1) 
$$
tends to $\omega_0$ , when the parameter $\varepsilon$ tends to $0$.
This shows that $\omega_0 \in \overline{\mathcal{L}(1,1,1)}$. 
The  missing details  can be found in  \cite[Appendix] {gavr16}. 
\endproof

The exceptional set $\mathcal Q_4$ might look not quite explicit, we investigate it in details bellow.

The space $\mathcal A_n$ of polynomial vector fields of degree  at most $n$  are identified to a vector space of dimension $(n+1)(n+2)$. On $\mathcal A_n$ acts the affine group 
$\textit{Aff}_2$
of affine transformations of $K^2$ (as usual $K=\R$ or $K=\C$), as well the multiplicative groupe $K^*$ corresponding to "change of time", $\dim \textit{Aff}_2 \times K^* = 7$. Therefore the minimal 
dimension of a component of the central set $\mathcal C_n$ is $7$. Such components, if exist, will be in a sense exceptional. 

In the quadratic case $n=2$ the dimension of the for components of $\C^2$ are easily found. For instance, in the case $\mathcal L(1,1,1)\subset \mathcal A_2$, and up to an affine changes of variables and time, one may suppose that the first integral is in the form $xy^\lambda (1-x-y)^\mu$. Therefore the dimension of $\mathcal{L}(1,1,1)$  is $2+7=9$ and the codimension is $3=12-9$. We find similarly that 
$\dim \mathcal L(2,1) = \dim \mathcal L(3) = 9$.

We describe now the last component $\mathcal Q_4$.
 Let $[x:y:z] $ be homogeneous coordinates in $\P^2$
\begin{align}
\label{pol2}
P_2(x,y,z)&= a_2(x,y) + a_1(x,y) z + a_0(x,y) z^2\\
\label{pol3}
P_3(x,y,z)& = b_3(x,y) +b_2(x,y) z + b_1(x,y)z^2 + b_0(x,y) z^3
\end{align}
be homogeneous polynomials in $x,y,z$ of degree $2$ and $3$. The function $$H=P_2^3/P_3^2$$ is therefore rational on $\P^2$ and induces a  foliation on $\P^2$
\begin{equation}
\label{p2}
3 P_3(x,y,z) d P_2(x,y,z)
-2 P_2(x,y,z) d P_3(x,y,z) =0 .
\end{equation}
 The corresponding affine foliation  on  the chart $\C^2$ defined by $z=1$
\begin{equation}
\label{aff2}
3 P_3(x,y,1) d P_2(x,y,1)
-2 P_2(x,y,1) d P_3(x,y,1) =0
\end{equation}
is of degree $ 4$. We may obtain a plane polynomial foliation of degree $ 2$ by imposing the following additional conditions. 

Suppose first, that the infinite line $\{z=0\}$ is invariant, that is to say (up to affine change) 
\begin{equation}
H(x:y:1)= \frac{a_2(x,y)^3}{  b_3(x,y)^2}  = 1 .
\end{equation}
This condition can be written as
$$
P_2(x,y,z)^3=P_3(x,y,z)^2 + O(z) .
$$
The foliation  (\ref{p2}) takes the form
$$
z[P(x,y,z)dx + Q(x,y,z)dy] + R(x,y,z) dz = 0
$$
where $\deg P, \deg Q \leq 3$, so (\ref{aff2}) is of degree $ 3$. If we further suppose that $z$ divides the homogeneous one form $3P_3dP_2-2P_2dP_3$ then (\ref{p2}) takes the form
$$
z^2[P(x,y,z)dx + Q(x,y,z)dy] + zR(x,y,z) dz = 0
$$
where $\deg P, \deg Q \leq 2$, so (\ref{aff2}) is a plane quadratic foliation. The condition that $z^2$ divides  $2P_3dP_2-3P_2dP_3$ can be written as
$$
P_2(x,y,z)^3=P_3(x,y,z)^2 + O(z^2)
$$
or equivalently
\begin{align}
\label{eqn1}
a_2(x,y)^3 &=   b_3(x,y)^2 \\
\label{eqn2}
3 a_2(x,y)^2 a_1(x,y) &= 2  b_3(x,y)  b_2(x,y) .
\end{align}
These polynomial relations can be further simplified by affine changes of the variables $x,y$. 
First,  (\ref{eqn1}) implies that $a_2$ is a square of a linear function in $x,y$ which we may supose equal to $x$, that is to say
$$
a_2(x,y)= x^2, \; b_3(x,y) = x^3 .
$$
The second condition (\ref{eqn2}) becomes $3x a_1= 2 b_2$ where we may put $a_1=2 y$, and hence
$$
a_1(x,y) = 2y,\;  b_2(x,y) = 3 xy .
$$
It is seen that the polynomial $P_3(x,y,1)$ has a real critical point  which we can put at the origin, so we shall also suppose that $b_1=0$. Using finally a "change of time" (the action of $K^*$) we assume
that $b_0=1$ while $a_0= \alpha \in K$ is a free parameter (modulus). The first integral takes therefore the form
\begin{equation}
\label{aps2}
H_\alpha(x,y) = \frac{ (x^2 + 2 y + \alpha)^3}{(x^3 + 3xy + 1)^2}
\end{equation}
with induced quadratic foliation
\begin{equation}
\label{aps3}
(- \alpha x^{2}  - 2 y^{2} -  \alpha y  + x) dx + (x y -  \alpha x + 1)dy .
\end{equation}
This is the exceptional co-dimension four component of  $Q_4$.   

The reader may check that the corresponding vector field
$$
x'= x y -  \alpha x + 1, \; y' =  \alpha x^{2}  +2 y^{2} +  \alpha y  - x
$$
has a Morse center at $x=1/\alpha, y=0$ which is moreover a usual real center 
 for $\alpha \in (\-1,0)$. The above computation is suggested by  \cite{lins14} where, however, the modulus $\alpha$ is wrongly fixed equal to $\alpha=\infty$). 
 The foliation on $\P^2$ corresponding to
 $$
 H_\infty(x,y) = \frac{ (x^2 + 2 y + 1)^3}{(x^3 + 3xy)^2}
 $$
 has two invariant lines $\{x=0\}$ and $\{z=0\}$, in contrast to the general foliation defined by $dH_\alpha(x,y) = 0$ which has only one invariant line $\{z=0\}$.
 We resume the above as follows
 \begin{proposition}
Every polynomial vector field having a rational first integral of the form
  $$H(x,y)=  \frac{(   a_0(x,y) + a_1(x,y)  + a_2(x,y) )^3}{( b_0(x,y)+  b_1(x,y) + b_2(x,y) + b_3(x,y))^2}
  $$
where the homogeneous polynomials $a_i, b_j$ of degrees $0\leq i \leq 2$, $0\leq j \leq 3$ are subject to the relations
\begin{align*}
a_2(x,y)^3 &=   b_3(x,y)^2 \\
3 a_2(x,y)^2 a_1(x,y) &= 2  b_3(x,y)  b_2(x,y) 
\end{align*}
is of degree two. The set of such quadratic vector fields form the irreducible component $\mathcal Q_4$ of the center set $\mathcal C_2$.
Up to an affine change of the variables $x,y$ the polynomial $H$ can be assumed in the form 
$H(x,y) = \frac{ (x^2 + 2 y + \alpha)^3}{(x^3 + 3xy + 1)^2}$ where $\alpha$ is a parameter.
  \end{proposition}
 
 We conclude this section with the following remarkable property of $\mathcal Q_4$.
 One may check that general rational function of the form $H(x,y)=P_2^3/P_3^2$, where $P_2,P_3$ are bi-variate polynomials of degree two and three,
 defines a pencil of genus four curves $\Gamma_t = \{(x,y):H(x,y)=t\}$ on $\C^2$. 
However, the special rational function $ H_\alpha$ (\ref{aps2}) defines an elliptic pencil, that is to say the level sets  $\Gamma_t = \{(x,y):H_\alpha(x,y)=t\}$ on $\C^2$
$$\{(x,y) \in \C^2 : H_\alpha(x,y)= const \}$$
are genus one curves, see 
 \cite{gail09}.

\subsection{An example : the central set of the polynomial Li\'enard equation}
\label{section4b}
Consider the following  polynomial Li\'enard equation
\begin{align}
\dot{x}  = y , \; \dot{y} = -q(x) - y p(x) \label{lie1}
\end{align}
in which the origin $(0,0)$ is an isolated singular point. Note that it is equivalent to the second order non-linear equation
$$
\ddot{x} + p(x) \dot{x} + q(x) = 0
$$
(a generalised Van der Pol oscillator). The description of the non-degenerate real analytic centers of (\ref{lie1}) is due to Cherkas \cite{cher72}, and in the polynomial case Christopher \cite{chri07} 
gave an  interpretation, which we formulate bellow.

The Li\'enard equation  (\ref{lie1}) induces a polynomial foliation
\begin{equation}
\label{lie2}
y dy + (q(x) + y p(x)) dx = 0 .
\end{equation}
Suppose that the following Polynomial Composition Condition (PCC) is satisfied

\emph{There exist polynomials $\tilde P, \tilde Q, W$ such that
\begin{equation} \tag{PCC} 
P= \tilde P\circ W, \;Q= \tilde Q \circ W 
\end{equation} 
where $P'(x)=p(x), Q'(x)=q(x))$  .} 

Then   the Li\'enard  foliation  takes the form
$$
ydy +   d\tilde Q(W) + y d\tilde P(W) = 0 .
$$
It is easy to see that if the origin $x=0$ is a Morse critical point of $W$, and if  $q'(0)>0$, then 
f the Li\'enard equation has a center at $(0,0)$.

The Theorem of Cherkas-Christopher can be formulated as follows
\begin{theorem}
The real polynomial Li\'enard equation (\ref{lie1}) has a non-degenerate real center if and only if $q'(0)>0$, and the polynomials $p, q$ satisfy the above Polynomial Composition Condition, where the real polynomial $W$ has a Morse critical point at the origin.
\end{theorem}
A self-contained proof can be found in  \cite{chri07} , it is based on the following simple  observation due to  Cherkas \cite[Lemma 1]{cher72}
\begin{lemma} 
The real analytic equation 
\begin{align}
\dot{x}  = y , \; \dot{y} = -x+ y \sum_{i=1}^\infty a_i x^i \label{cher}
\end{align}
has a center at the origin, if and only if $a_{2j}= 0, \forall j\geq 1$.
\end{lemma}
Indeed, the truncated equation
\begin{align}
\label{truncated}
\dot{x}  = y , \; \dot{y} = -x+ y \sum_{j=0}^\infty a_{2j+1} x^i
\end{align}
has a center.  In a sufficiently small neighbourhood of the origin, (\ref{truncated}) is rotated with respect to the vector field (\ref{cher}), unless $a_{2j}= 0, \forall j\geq 1$. 
The final argument of Christopher is to use the L\"uroth theorem, to the deduce the PCC condition.
This topological argument of Cherkas does not apply in a complex domain.
We shall prove, however, the following more general
\begin{theorem}
\label{th2}
The (possibly complex) polynomial Li\'enard equation (\ref{lie1}) has a Morse critical point at the origin, if and only if the polynomials $p, q$ satisfy the above Polynomial Composition Condition, 
and $W$ has a Morse critical point at te origin. 
\end{theorem}
The condition (\ref{lie2})  to have a linear Morse critical point implies $q(0)=0, q'(0)\neq 0$ and $p(0) = 0$. Therefore the polynomial 
$\frac 12 y^2 + Q(x)$ has a Morse critical point at the origin and there exists a local bi-analytic  change of the variable $x\to X$ such that
$\frac 12 y^2 + Q(x(X))= \frac12 (y^2+ X^2)$. Thus
\begin{align}
y dy + (q(x) + y p(x)) dx = \frac12  (y^2+ X^2) + y d P(x(X)) .
\end{align}
We expand 
$$
 d P(x(X)) = ( \sum_{i=1}^\infty a_i X^i ) dX
$$
and obtain the equivalent foliation
\begin{equation}
\label{lie3}
\frac12 d(y^2+X^2)  +   y ( \sum_{i=1}^\infty a_i X^i ) dX = 0 .
\end{equation}
By analogy to the Cherkas Lemma we shall prove
\begin{lemma}
\label{l2}
The foliation (\ref{lie3}) has a Morse critical point at the origin if and only if $a_{2j}= 0, \forall j\geq 1$.
\end{lemma}
\proof
After rescaling $(X,y) \mapsto \varepsilon (X,y)$ the foliation takes the form
\begin{align}
\mathcal F _\varepsilon : \frac12 d(y^2+X^2)  +   y ( \sum_{i=1}^\infty \varepsilon ^{i}a_i X^i ) dX = 0
\end{align}
and it suffice to prove that for sufficiently small $\varepsilon$ it has a Morse critical point. Note first that the truncated equation 
\begin{align}
\mathcal F _\varepsilon^t : \frac12 d(y^2+X^2)  +   y ( \sum_{j=0}^{\infty} \varepsilon ^{2j+1}a_{2j+1} X^{2j+1} ) dX = 0
\end{align}
is a pullback of
\begin{align}
 \pi^*\mathcal F _\varepsilon^t  : \frac12 d(y^2+\xi)  +   y ( \sum_{j=0}^{\infty} \varepsilon ^{2j+1}a_{2j+1} d \frac{\xi^{j+1} }{2j+2} = 0
\end{align}
under the map $\pi : (X,y) \mapsto (\xi, Y), \xi = X^2$. The foliation $ \pi^*\mathcal F _\varepsilon^t$ is regular at the origin and has a first integral
$$
\frac12(y^2+\xi) + O(\varepsilon)
$$
where $O(\varepsilon)$ is analytic in $\varepsilon, \xi, Y$, and vanishes as $\varepsilon = 0$. Thus $\mathcal F _\varepsilon^t$ has a first integral
$$
H_\varepsilon(x,y) = \frac12 (y^2+X^2) + O(\varepsilon)
$$
where $O(\varepsilon)$ is analytic in $\varepsilon, X^2, Y$, and vanishes as $\varepsilon = 0$. This also shows that the origin is a Morse critical point.

As $H_\varepsilon$ is a first integral of the truncated foliation $\mathcal F _\varepsilon^t$ then 
for every fixed $\varepsilon$ we have
$$
 (1+ O(\varepsilon)) d H_\varepsilon(x,y) =  \frac12 d(y^2+X^2)  +   y ( \sum_{j=0}^{\infty} \varepsilon ^{2j+1}a_{2j+1} X^{2j+1} ) dX .
$$

Suppose now that for some $j \geq 1$ , $a_{2j} \neq 0$ and let $j=k$ be the smallest integer with this property. We have 
\begin{align}
\mathcal F _\varepsilon :  (1+ O(\varepsilon)) d H_\varepsilon(x,y) +  \varepsilon^{2k} ya_{2k}X^{2k}dX + O(\varepsilon^{2k+1})dX = 0
\end{align}
where by abuse of notations $ O(\varepsilon^{2k+1})$ denotes an analytic funstion in $X,y,\varepsilon$ which is divisible by $\varepsilon^{2k+1}$.
The origin is a Morse critical point if and only if the holonomy map of the two separatrices of $\mathcal F _\varepsilon$ at the origin, are the identity maps.
The holonomy map will be evaluated by  the usual Poincar\'e-Pontryagin-Melnikov formula. The separatrices are tangent to the lines $y \pm i X = 0$. 
We take a cross-section to one of the  separatrices, parameterised by the restriction of $H_\varepsilon(x,y)$ on it. Let
$$
\gamma_\varepsilon(h) \subset \{(x,y) : H_\varepsilon(x,y) = h \}
$$ 
be a continuous family of closed loops vanishing at the origin as $h \to 0$. The holonomy map of $\mathcal F _\varepsilon$ , corresponding to this closed loop is
\begin{align*}
h \mapsto & h + \frac{\varepsilon^{2k}}{1+O(\varepsilon)} ( \int_{\gamma_\varepsilon(h)}  ya_{2k}X^{2k} dX + O(\varepsilon) dX) \\
&=   h+ \varepsilon^{2k}  \int_{\gamma_\varepsilon(h)}  ya_{2k}X^{2k} dX + O(\varepsilon^{2k+1}) dX \\
& = h +  \varepsilon^{2k} a_{2k}  \int_{\gamma_0(h)}  y X^{2k} dX + O(\varepsilon^{2k+1}) d
\end{align*}
where 
$$
\gamma_0(h) \subset  \{(x,y) : H_0(x,y) = h \} =  \{(x,y) :   \frac12 ( y^2+X^2 ) = h \} .
$$
By homogeneity of the polynomials
$$
 \int_{\gamma_0(h)}  y X^{2k} dX = h^{k+1}  \int_{\gamma_0(1)}  y X^{2k} dX
$$
As the  homology of the algebraic curve $\{ y^2+X^2 ) = 2h \}$ has one generator we can suppose that this generator is just the real circle $\gamma_0(1)$ is 
just the real circle $\{ (y,X)\in \R^2 : y^2+X^2 = 2 \}$ and in this case 
$$
 \int_{\gamma_0(1)}  y X^{2k} dX = \iint_{y^2+X^2 \leq 2 }X^{2k} dX dy \neq 0 .
$$
We conclude that if the holonomy map is the identity map, then $a_{2k} = 0$ which is the desirable contradiction. Lemma \ref{l2} is proved.
\endproof
\proof[Proof of Theorem \ref{th2}]
Assuming that the Li\'enard equation has a Morse critical point, and hence $Q(x)$ has a Morse critical point at the origin, denote $x_1(h), x_2(h)$ the two roots of the polynomial $Q(x)-h$ which vanish at $0$ as $h$ tends to $0$. We have obviously that $X(x_1(h)) = - X(x_2(h))$.
By Lemma \ref{l2} the analytic function $P(x(X))$ is even in $X$, and hence 
$P(x_1(h))\equiv P(x_2(h))$. Following an idea of Christopher (already used at the end of section \ref{section3}), consider now the field $\mathcal C  \subset \C(x)$ formed by all rational functions $R=R(x)\in\C(x)$ satisfying the identity
$$
R(x_1(h))\equiv  R(x_2(h)) .
$$
According to the L\"uroth theorem, every subfield of $\C(x)$ is of the form $\C(W)$ for some rational function $W=W(x)$. Thus we have 
$\mathcal C = \C(W)$ where $P,Q \in \mathcal C$. This already implies that $P, Q$ satisfy (PCC) which completes the proof of the Theorem.
\endproof

\subsection{Abel equations with Darboux type first integral}
\label{section43}

The polynomial Li\'enard equation
\begin{align}
\label{lienard}
\dot{x}  = y , \; \dot{y} = -q(x) - y p(x)
\end{align}
with associated foliation $y dy + (q(x) + y p(x)) dx = 0$,
after the substitution $y\to 1/y$,  becomes the following
Abel equation
\begin{align}
\label{abel}
\frac{dy}{dx} =  y^2 p(x) + y^3 q(x) .
\end{align}
Equivalently, we consider the foliation   
\begin{equation}
\label{abel243}
 dy = ( y^2 p(x) + y^3 q(x)) dx .
\end{equation}
The classification of  Morse critical points of the Li\'enard equation  (\ref{lie1}) obtained in section \ref{section4b}
 suggests that a similar claim would hold true for the scalar Abel equation (\ref{abel243}). This is the content of the following \\
 
\emph{Composition Conjecture  \cite[p.444]{bry10}. 
The Abel equation (\ref{abel243}) has a center at the solution $y=0$ along some fixed interval $[a,b]$ if and only if the following (PPC) condition holds true}
\begin{equation}
P= \tilde P\circ W, \;Q= \tilde Q \circ W, \; W(a)=W(b) . \tag{PPC}
\end{equation}
\vspace{1cm}\\
Note that the Cherkas-Christopher theorem is for non-degenerate centers. The Composition Conjecture missed the possibility for the Abel or Li\'enard equations to have a Darboux type first integral, with resonant saddle point and characteristic ratio $p:-q$  (instead of a non-generate center with $1:-1$ ratio). Incidentally, Li\'enard equations with a Darboux type first integral will produce counter-examples to the Composition Conjecture, which is the subject of the present section. We explain in this context the recent counter-example 
of Gin\'e, Grau and Santallusia \cite{ggs17}.

The method of constructing such systems is based on the example of the co-dimension four center set $\mathcal Q_4$ for quadratic system, as explained in  section \ref{section42}. The general method is outlined in the Appendix.

Let
 \begin{align*}
P_2 &= a_0(x) + a_1(x)y + a_2(x) y^2\\
Q_2 & = b_0(x)+ b_1(x)y + b_2(x) y^2
\end{align*}
where $a_i, b_j$ are polynomials, such that $P_2^p=Q_2^q + O(y^3)$, where $p,q$ are positive relatively prime integers. 
This implies that the corresponding one-form $$ p Q_2 dP_2 -  q P_2 dQ_2$$ is divisible by $y^2$, and then the associated reduced foliation (after division by $y^2$)
is of degree two in $y$, and moreover $\{y=0\}$ is a leaf. Therefore the foliation is defined as
\begin{align}
\label{eq1}
(r_1y+r_2) dy & = y(r_3 y + r_4) dx = 0, r_i \in \mathbb C[x] 
\end{align}
where 
\begin{align*}
r_1 &= 2(p-q)a_2b_2\\
r_2 & = (p-2q)a_1b_2 -(q-2p)b_1 a_2 \\
r_3 &= pa_2'b_2-qb_2' a_2\\
r_4 & = p a_1'b_2 - qb_1' a_2
\end{align*}

Note that if  $a_2 = const.\neq 0$, $b_2 = const.\neq 0$ the foliation takes the Li\'enard form
\begin{align}
\label{eq2}
(r_1y+r_2) dy & = y r_4 dx , r_1=const.
\end{align}
 Of course, it is not clear, whether such polynomials exist.
To verify this we have to solve the equation 
\begin{align*}
( a_0(x) + a_1(x)y + a_2(x) y^2)^p &= (b_0(x)+ b_1(x)y +  b_2(x) y^2)^q \mod y^3 
\end{align*}
assuming that $a_i(x), b_j(x)$ are polynomials, and $a_2, b_2$ are constants. A first condition is given by $$a_0^p=b_0^q$$ which implies
\begin{align*}
( 1 + \frac{a_1(x)}{a_0(x)} y + \frac{a_2(x)}{a_0(x)} y^2)^p &= (1+ \frac{b_1(x)}{b_0(x)}y +  \frac{b_2(x) }{b_0(x)}y^2)^q \mod y^3 
\end{align*}

or equivalently
\begin{align*}
p  \frac{a_1(x)}{a_0(x)} &= q \frac{b_1(x)}{b_0(x)} \\
p  \frac{a_2(x)}{a_0(x)} + \frac{p(p-1)}{2} ( \frac{a_1(x)}{a_0(x)} )^2 &= q \frac{b_2(x)}{b_0(x)} + \frac{q(q-1)}{2}(\frac{b_1(x)}{b_0(x)} )^2 .
\end{align*}
Thus $a_i, b_j$ are polynomials which satisfy the following redundant system of equations
\begin{align*}
a_0(x)^p & =b_0(x)^q \\
p  \frac{a_1(x)}{a_0(x)} &= q \frac{b_1(x)}{b_0(x)} \\
p\frac{a_2(x)}{a_0(x)}  -  \frac{p}{2} ( \frac{a_1(x)}{a_0(x)} )^2 &= q \frac{b_2(x)}{b_0(x)} - \frac{q}{2}(\frac{b_1(x)}{b_0(x)} )^2 .
\end{align*}
It follows that for some polynomial $R$, 
$$
a_0(x)=R(x)^q, \; b_0 = R(x)^p
$$
and moreover 
$$
p a_2R(x)^{-q} - q b_2R(x)^{-p}
$$
is a square of a rational function, where we recall that $a_2=const., b_2 = const.$. It is easy to check that this is only possible if, say $p<q$, and $ p=2k-1, q= 2k$ for an integer $k\geq 1$.  
With this observation the analysis of the system is straightforward and is left to the reader. We formulate the final result in the following 
\begin{theorem}
For every integer $k\geq 1$ and polynomial $r(x)$ the function
\begin{align*}
H(x,y) = & \frac{[(1-r(x)^2)^{2k} + 2k r(x) (1-r(x)^2)^k y + k y^2]^{2k-1}}{[(1-r(x)^2)^{2k-1} + (2k-1) r(x) (1-r(x)^2)^{k-1} y + \frac{2k-1}{2} y^2]^{2k}}
\end{align*}
is the first integral of a Li\'enard type equation of the form
\begin{align*}
\frac{dx}{dt}&=-y+r_2(x),\\
\frac{dy}{dt}&=y r_4(x),
\end{align*}
for suitable polynomials $r_2, r_4$.
\end{theorem}
It is clear that the above Li\'enard system is a polynomial pull back under $x \to r(x)$ of a simpler\emph{ master Li\'enard system } with first integral
\begin{align}
\label{master}
H_k(x,y) = & \frac{[(1-x^2)^{2k} + 2k x (1-x^2)^k y + k y^2]^{2k-1}}{[(1-x^2)^{2k-1} + (2k-1) x (1-x^2)^{k-1} y + \frac{2k-1}{2} y^2]^{2k}}
\end{align}
which can not be further reduced.

To the end of the section we consider in more detail the simplest cases $k=1$ and $k=2$. For $k=1$ we get
$$
H_1(x,y)=
\frac{ (1-x^2)^2 + 2x(1-x^2) y + y^2}{(  1-x^2+ xy +\frac12 y^2 )^2} .
$$
which is the first integral of the following cubic 
 Li\'enard equation :

\begin{equation}
\begin{array}{l}
\frac{dx}{dt}=y+3x(1-x^2)\\
\frac{dy}{dt}=-y(1+3x^2)
\end{array}
\end{equation}
The characteristic ratios of the singular points $(0,0), (\pm 1,0)$ are equal to $3:-1$  and $3:-4$
and the characteristic values of
$(+\frac{1}{\sqrt{3}},-\frac{1}{\sqrt{3}})$ and $(-\frac{1}{\sqrt{3}},+\frac{1}{\sqrt{3}})$  equal $-1$ (so we have integrable cubic saddles, presumably new).  
The Li\'enard transformation $Y= 2y+3x(1-x^2)$ transforms the above to the standard form
$$
\frac{dx}{dt} = Y, \frac{dY}{dt}= p(x) + q(x) Y  \Leftrightarrow Y \frac{dY}{dx} = p(x) + q(x) Y .
$$
or also to the Abel type equation
$$
\frac{d}{dx} (1/Y) = - q(x) (\frac{1}{Y})^2 - p(x)  (\frac{1}{Y})^3
$$
with respect to the variable $z=1/Y$.

Assume that $k=2$, the first integrals takes the form 
 $$
H_2(x,y)=\frac{ (y^2 - 2 xy(1-x^2)^2  + \frac12 (1-x^2)^4)^3}{( y^2 -  2xy (1-x^2) + \frac23 (1-x^2)^3)^4} 
 $$
while the corresponding foliation of Li\'enard type is defined by
\begin{equation}
\label{eql1}
(15x^4 - 6x^2 - 1)y dx - ( (5x^2-3)(x^2-1)x  + y) dy = 0 .
\end{equation}
Namely, the Li\'enard transformation 
$$
y \to - y - (5x^2-3)(x^2-1)x
$$
transforms the equation (\ref{eql1}) to 
\begin{equation}
\label{eql2}
(q(x)+ p(x) y )dx + y dy = 0 
\end{equation}
or equivalently to
\begin{align*}
\frac{dx}{dt}&=-y,\\
\frac{dy}{dt}&=q(x)+p(x)y,
\end{align*}
where
\begin{align*}
p(x) &= 2   (20 x^{4} - 15 x^{2} + 1) \\
q(x) &= x   (x - 1)   (x + 1)   (5 x^{2} - 3)  (15 x^{4} - 6 x^{2} - 1) .
\end{align*}
The first integral $H_2$ takes the form
$$
\frac{(y^2-
8x(1-x^2)(x^2 - \frac12)y+(1- x^2)^2(x^2 - \frac12)(15x^4 - 6x^2 -1))^3}{( y^2-2x(1- x^2)(5x^2 - 2)y+\frac13 (1- x^2)^2(5x^2 -2)(15x^4 - 6x^2 -1))^4} .
$$
However, after the substitution $1\to 1/y$, the above Li\'enard equation is equivalent to the Abel equation
\begin{equation}
\label{ggs2}
\frac{dy}{dx} = p(x)y^2 + q(x) y^3
\end{equation}
with Darboux type first integral
$$
H= \frac{y^2[1-
8x(1-x^2)(x^2 - \frac12)y+(1- x^2)^2(x^2 - \frac12)(15x^4 - 6x^2 -1)y^2]^3}{[ 1-2x(1- x^2)(5x^2 - 2)y+\frac13 (1- x^2)^2(5x^2 -2)(15x^4 - 6x^2 -1)y^2]^4} .
$$

\begin{theorem}[\cite{ggs17}]
The Abel equation (\ref{ggs2}) has a center at $y=0$ along the interval $[-1,1]$ but this center is not universal.
\end{theorem}
\proof
As $y=0$ is a solution of (\ref{ggs2}) then for all sufficiently small $|\varepsilon|$ the solution $y=y(x)$ with initial condition 
 $y(-1)=\varepsilon\neq 0$  exists along the compact interval $[-1,1]$. 
 The identity $H(\pm 1,y)= y^2$ shows that $y(1)=\varepsilon$ or $y(1)= - \varepsilon$ and it is easy to check that in fact $y(1)=\varepsilon$, Indeed,  for real $\varepsilon$ the solution $y(x)$ does not vanish and hence it has the same sign as $\varepsilon$.
Therefore the transport map along the interval $[-1,1]$ is the identity map, and the Abel equation has a center at the solution $y=0$. 

The polynomials $P=\int p$ and $Q= \int q$ are of degrees   $5$  and $10$. 
Therefore if they  had a common non-trivial composition factor, then the factor would be $P$ and $Q= \tilde Q \circ P$ for suitable quadratic polynomial $\tilde Q$. It follows that $p=P'$ divides $q=Q'$ which is obviously not true. Thus $P,Q$ can not have a common composition factor, and  (by the Brudnuy's theorem) the center is not universal. 
\endproof 

\section{Appendix}

Let $K$ be a field and $A= K[[y]](x)$ be the $K(x)$ algebra of formal power series in $y$, with coefficients in the field of rational functions.
To each fixed  pair of mutually prime positive integers $p,q$\footnote{the case when $pq<0$ is treated in a similar way} we associate the functional equation 
\begin{align}
\label{pq}
P^p=Q^q, \; P,Q \in A .
\end{align}
\begin{proposition}
Every solution of the functional equation (\ref{pq}) has the form $P=R^p, Q= R^q$ where 
$$
R(x,y) = a_0(x) (1+ O(y))
$$
$a_0(x)$ is a rational function in $x$, and $ O(y) \in A$ is divisible by $y$.
\end{proposition}
The proof is straightforward.
 $A$ allows an ascending filtration with respect to the powers of $y$,  $K\subset \dots \subset A_n \subset A_{n+1} \subset \dots $  where
 $$
 A_n = \{ \sum_{i=0}^n a_i(x) y^i : a_i \in K(x) \} .
 $$
 For   $P\in A$ we denote $P_n$ the projection of $P$ on $A_n$, that is to say
 $$
 P= \sum_{i\geq 0} a_i y^i, \; P_n = \sum_{i=0}^n a_i y^i .
 $$
 Obviously $A_n\subset K[y](x)$ and to 
 the pair of polynomials $P_n, Q_n \in K[x,y] $ we associate the logarithmic polynomial differential one-form
 $$
 \omega = P_n Q_n d \log \frac{P_n^p}{ Q_n^{q}} = p Q_n d P_n - q P_n d Q_n .
 $$
If $\omega = \alpha dx + \beta dy$, $\alpha, \beta \in K[y](x)$,  then the degree $\deg_y \omega $ of  $\omega$ (with respect to $y$) is
$$
\max \{ \deg_y \alpha, \deg_y \beta + 1 \}  .
$$
 For general $P, Q \in A$ of degree 
 the degree of the associated $\omega$ equals $2n$, however
 \begin{lemma}
 \label{masterlemma}
 If the formal series $P, Q$ satisfy the functional equation (\ref{pq}), then the associated one-form $\omega =  p Q_n d P_n - q P_n d Q_n$ is divisible by $y^n$.
 The resulting reduced rational differential one-form $\omega_{red} = \omega/y^n$ is of degree $n$ with respect to $y$, and moreover vanishes along the line $\{ y= 0 \}$.
   \end{lemma}
   \proof
   Assume that $P^p=Q^q$. Without loss of generality we may suppose that $y$ does not divide $P$ or $Q$. 
    As $P=P_n + O(y^{n+1}), Q= Q_n + O(y^{n+1})$ then
   $
   P_n^p=Q_n^q + O(y^{n+1}) 
   $
 and hence $p P_n^{p-1} d P_n = q Q_n^{q-1} d Q_n + d O(y^{n+1}) $. It follows  that $y^n$ divides the differential form
\begin{align*}
 d \frac{P_n^p}{Q_n^q} = d (\frac{P_n^p}{Q_n^q} - 1) = d \frac{P_n^p-Q_n^q}{Q_n^q} 
 \end{align*}
 and the reduced differential form  $d \frac{P_n^p}{y^nQ_n^q} $ vanishes along $y=0$. Therefore the same holds true for the
one-form $p Q_n d P_n - q P_n d Q_n$.
\endproof

\section*{Acknowledgement}
The author is obliged to Jean-Pierre Fran\c{c}oise for the illuminating discussions and suggestions.

\bibliographystyle{plain}
\def\cprime{$'$} \def\cprime{$'$} \def\cprime{$'$} \def\cprime{$'$}
  \def\cprime{$'$} \def\cprime{$'$} \def\cprime{$'$}

\end{document}